\documentclass{amsart}

\usepackage{amssymb,color}
\usepackage{amsfonts}
\usepackage{amsmath}
\usepackage{euscript}
\usepackage{enumerate}
\usepackage{graphics}

\newtheorem{theorem}{Theorem}[section]
\newtheorem{lemma}[theorem]{Lemma}
\newtheorem{note}[theorem]{Note}
\newtheorem{question}[theorem]{Question}
\newtheorem{prop}[theorem]{Proposition}
\newtheorem{cor}[theorem]{Corollary}
\newtheorem{exa}[theorem]{Example}

\newtheorem*{Theorem1'}{Theorem 1'}

\theoremstyle{definition}

\theoremstyle{remark}

\numberwithin{equation}{section}

\newcommand \g{{\mathfrak g}}

\newcommand  \s{{\mathfrak s}}

\newcommand \gl{{\mathfrak {gl}}}
\renewcommand \sl{{\mathfrak {sl}}}

\newcommand \GL{{\mathrm {GL}}}

\newcommand \F{{F_\ell}}
\newcommand \C{{\Phi}}
\newcommand \EF{{\mathbb F}}

\begin{document}

\title{\small Generalized Artin-Schreier polynomials}

\author{ N. H. Guersenzvaig}
\address{A\!v. Corrientes 3985 6A, (1194) Buenos Aires, Argentina}
\email{nguersenz@fibertel.com.ar}
%\thanks{The first author was supported in part by CONICET and SECYT-UNC grants.}

\author{Fernando Szechtman}
\address{Department of Mathematics and Statistics, University of Regina, Canada}
\email{fernando.szechtman@gmail.com}
\thanks{The second author was supported in part by an NSERC discovery grant}

%    General info

%\date{January 1, 2001 and, in revised form, June 22, 2001.}


\begin{abstract} Let $F$ be a field of prime characteristic $p$
containing $\EF_{p^n}$ as a subfield. We refer to
$q(X)=X^{p^n}-X-a\in F[X]$ as a generalized Artin-Schreier
polynomial. Suppose that $q(X)$ is irreducible and let $C_{q(X)}$
be the companion matrix of $q(X)$. Then $ad\, C_{q(X)}$ has such
highly unusual properties that any $A\in\gl(m)$ such that $ad\, A$
has like properties is shown to be similar to the companion matrix
of an irreducible generalized Artin-Schreier polynomial.

We discuss close connections with the decomposition problem of the tensor product
of indecomposable modules for a 1-dimensional Lie algebra over a field of characteristic $p$,
the problem of finding an explicit primitive element for every intermediate field of the Galois
extension associated to an irreducible generalized Artin-Schreier polynomial, and the
problem of finding necessary and sufficient conditions for the irreducibility of a family of polynomials.
\end{abstract}

\maketitle

\section{Introduction}

If $\g$ is a finite dimensional semisimple Lie algebra over an
algebraically closed field of characteristic 0 then the finite
dimensional indecomposable (i.e., irreducible) $\g$-modules are
well understood (see \cite{H}). However, the problem of
classifying the finite dimensional indecomposable modules of an
arbitrary finite dimensional Lie algebra over any given field is
in general unattainable. This is explained in \cite{GP} when $\g$
is abelian with $\dim(\g)>1$ and, more generally, in \cite{M} for
virtually any Lie algebra over an algebraically closed field of
characteristic 0 that is not semisimple or 1-dimensional.
Nevertheless, significant progress can be made by restricting
attention to certain types of indecomposable modules. In a sense,
the simplest type of indecomposable module other than irreducible
is a uniserial module, i.e., one admitting a unique composition
series. A systematic study of uniserial modules for perfect Lie
algebras of the form $\s\ltimes V$, where $\s$ is complex
semisimple and $V$ is a non-trivial irreducible $\s$-module, is
carried out in \cite{CS}, which culminates with a complete
classification of uniserial modules for $\sl(2)\ltimes V(m)$, with
$m>0$.

The problem of classifying all uniserial modules of an abelian Lie
algebra over an arbitrary field $F$ is studied in \cite{CS2}. The
classification is achieved under certain restrictions on $F$, as
this subtle problem is intimately related to a delicate and
sharpened version of the classical theorem of the primitive
element. Such versions have been studied by Nagell \cite{N1} and
\cite{N2}, Kaplansky \cite{K}, Isaacs \cite{Is} and many other
authors (see \cite{CS2} and references therein). Suppose $F$
contains a copy of $\EF_{p^n}$ as subfield, where $p$ is a prime
and $n\geq 1$. Refer to $q(X)=X^{p^n}-X-a\in F[X]$ as a
generalized Artin-Schreier polynomial. In \cite{CS2}, Example 2.3,
irreducible generalized Artin-Schreier polynomials are used to
exhibit the limitations of some of the modified versions of the
theorem of the primitive element.

In \cite{CS3} attention is focused on the classification of
uniserial and far more general indecomposable modules over a
family of 2-step solvable Lie algebras $\g$. A complete
classification is achieved when the underlying field $F$ has
characteristic 0. The properties of the operator
$ad\,A:\gl(m)\to\gl(m)$, given by $B\mapsto [A,B]$, where
$A\in\gl(m)$ is the companion matrix of a polynomial $q\in F[X]$,
play a significant role in this classification. %Let $F$ be an
%arbitrary field  and let $M_m(F)$ the associative algebra of all
%$m\times m$ matrices over $F$. This becomes a Lie algebra, denoted
%by $\gl(m,F)$ or simply $\gl(m)$, under the usual bracket
%$[A,B]=AB-BA$. Each $A\in\gl(m)$ gives rise to the linear map
%$ad\, A:\gl(m)\to\gl(m)$, given by $B\mapsto [A,B]$.
If $F$ has prime characteristic $p$ and $q\in F[X]$ is an
irreducible generalized Artin-Schreier polynomial with companion
matrix $C_q$, then $ad\; C_q$ has exceptional properties which are
used in \cite{CS3}, Note 3.18, to furnish examples of uniserial
$\g$-modules that fall outside of the above classification.
%and in \cite{CS5}, \S 4,
%to produce a correct classification of uniserial $\g$-modules when $\chr(F)=p$ but $F$ is not
%algebraically closed.

Some of the easily verified properties enjoyed by the companion matrix $C_q\in\gl(p)$ of an irreducible classical Artin-Schreier polynomial, are as follows:

$\bullet$ All eigenvalues of $ad\; C_q$ are in $F$;

$\bullet$ The eigenvalues of $ad\; C_q$ form a subfield of $F$;

$\bullet$ The centralizer of $C_q$ is a subfield of $M_p(F)$;

$\bullet$ All eigenvectors of $ad\; C_q$ are invertible in
$M_p(F)$;

$\bullet$ All eigenspaces of $ad\; C_q$ have the same dimension;

$\bullet$ $ad\; C_q$ is diagonalizable with minimal polynomial
$X^p-X$.

In this paper we find all matrices $A\in\gl(m)$ such that $ad\,
A:\gl(m)\to\gl(m)$ has like properties.

\begin{theorem}\label{uno} Let $F$ be a field and let $A\in \gl(m,F)$. Then

(C1) All eigenvalues of $ad\, A$ are in $F$,

(C2) The eigenvalues of $ad\, A$ form a subfield of $F$,

(C3) The centralizer of $A$ is a subfield of $M_m(F)$,

\noindent if and only if

(C4) $F$ has prime characteristic $p$,

(C5) $A$ is similar to the companion matrix of a monic irreducible
polynomial $h\in F[X]$ of degree $m$,

(C6) If $q\in F[X]$ is the separable part of $h$, i.e.,
$h(X)=q(X^{p^e})$, $e\geq 0$, and $q$ is separable, then
$q=X^{p^n}-X-a$, where $a\in F$, $n\geq 1$, and $\EF_{p^n}$ is a
subfield of $F$.

Moreover, if (C4)-(C6) hold then: $q$ is irreducible; the subfield
of $F$ formed by the eigenvalues of $ad\, A$ is precisely
$\EF_{p^n}$; all eigenvectors of $ad\, A$ are invertible in
$M_m(F)$; all eigenspaces of $ad\, A$ have dimension $p^{n+e}$;
$ad\, A$ is diagonalizable if and only if $h$ itself is separable;
the invariant factors of $ad\, A$ are
$$
X^{p^{n+e}}-X^{p^e},\dots,X^{p^{n+e}}-X^{p^e},\text{ with
multiplicity }p^{n+e},
$$
so, in particular, the minimal polynomial of $ad\, A$ is
$X^{p^{n+e}}-X^{p^e}$.
\end{theorem}

The most challenging part of the proof of Theorem \ref{uno} is to
find the invariant factors of $ad\, A$, as this is depends on the solution
to the following problem.

Let $L=\langle x\rangle$ be a 1-dimensional Lie algebra over a
field $F$ and let $V$ and $W$ be indecomposable $L$-modules of
respective dimensions $n$ and $m$ upon which $x$ acts with at
least one eigenvalue from $F$.

\noindent{\sc Question.} How does the $L$-module $V\otimes W$ decompose as the direct sum of indecomposable $L$-modules?

When $F$ has characteristic 0 we may derive an answer from the
Clebsch-Gordan formula by imbedding $L$ into $\sl(2)$. A direct
computation in the complex case already appeared in \cite{Ro} in
1934. The results in characteristic 0 fail, in general, in prime
characteristic $p$, which is the case we require. The analogue
problem for a cyclic $p$-group when $F$ has prime characteristic
$p$ was solved by B. Srinivasan~\cite{S}. Her solution is of an
algorithmic nature. Since then several algorithms have appeared.
We mention \cite{Ra}, \cite{Re} and, most recently, \cite{I},
although the literature is quite vast on this subject. For
information on the decomposition of the exterior and symmetric
squares of an indecomposable module of a cyclic $p$-group in prime
characteristic $p$ see \cite{GL}. What we need to be able to
compute the invariant factors of $ad\, A$ in Theorem \ref{uno} is
a closed formula for the decomposition of the $L$-module $V\otimes
W$ when $n\leq m=p^e$, $e\geq 0$. This is achieved in \S\ref{ma4}.

The proof of Theorem \ref{uno} is given in \S\ref{ma3}, making use
of preliminary results from \S\ref{sec:pre} and \S\ref{ma4}.

In \S\ref{alfaH} we give an application to Galois theory of the
polynomials appearing in Theorem \ref{uno}. Let $F$ be a field of
prime characteristic $p$ containing $\EF_{p^n}$ as a subfield and
suppose that $q(X)=X^{p^n}-X-a\in F[X]$ is irreducible. Let $K/F$
be the corresponding Galois extension. Here $K=F[\alpha]$, where
$\alpha\in K$ is a root of $q(X)$. Given an arbitrary intermediate
field $E$ of $K/F$ we find a primitive element $\alpha_E$ such
that $E=F[\alpha_E]$. We actually give a recursive formula to
write $\alpha_E$ as a polynomial in $\alpha$ with coefficients in
$\EF_{p^n}$. This is achieved by means of the so-called Dickson
invariants, discovered by L. E. Dickson \cite{D} in 1911.

Finally, in \S\ref{hector} we discuss the actual existence of
irreducible polynomials $q(X)=X^{p^n}-X-a\in F[X]$, with $F$ as in
the previous paragraph. As explained in \S\ref{alfaH}, if $n>1$
and $q(X)$ is irreducible then $a$ must be transcendental over
$\EF_p$. This fact, together with the more general polynomials
$q(X^{p^e})$ considered in Theorem \ref{uno}, lead us to study the
irreducibility of polynomials of the form
$$h(X)=X^{p^{n+e}}-X^{p^e}-g(Z^r)\in F[X],$$
where $X$ and $Z$ are algebraically independent elements over an
arbitrary field $K$ of prime characteristic $p$, $n>0$, $r>0$,
$e\geq 0$, $F=K(Z)$, and $g(Z)\in K[Z]$ is a non-zero polynomial
of degree relatively prime to $p$. Using results from \cite{MS}
and \cite{G}, we obtain in \S\ref{hector} necessary and sufficient
conditions for the irreducibility of $h(X)$. In particular,
$X^{p^n}-X-g(Z^r)\in F[X]$ is irreducible for any $n>0$, $r>0$ and
non-zero $g\in K[Z]$ whose degree relatively prime to $p$. This
limitation on $\deg(g)$ is needed, as the example
$X^{p^n}-X-(Z^{p^n}-Z)$ shows.

\section{Eigenvalues}
\label{sec:pre}

Let $F$ be a field. For $A\in M_m(F)$ let $\chi_A$ and $\mu_A$
denote the characteristic and minimal polynomials of $A$. If $b\in
F$ is an eigenvalue of $A$ we write $E_b(A)$ for the corresponding
eigenspace. If $B\in M_m(F)$ we write $A\sim B$ whenever $A$ and
$B$ are similar. The companion matrix to a monic polynomial $g\in
F[X]$ of degree $m$ will be denoted by $C_g$.

\begin{lemma}\label{l1} Let $A\in M_m(F) $ and let $C$ be the centralizer of
$A$ in $M_m(F)$. Then $C$ is a subfield of $M_m(F)$ if and only if
$A$ is similar to the companion matrix of a monic irreducible
polynomial in $F[X]$ of degree $m$.
\end{lemma}

\begin{proof} If $A$ is similar to the companion matrix of a monic
polynomial of degree $m$ -necessarily $\mu_A$- it is well-known
\cite{J}, \S 3.11, that $C=F[A]$. If, in addition, $\mu_A$ is
irreducible, then $C=F[A]\cong F[X]/(\mu_A)$ is a field.

Assume that $C$ is a field. Then $K=F[A]$ is a field, so $\mu_A$
is irreducible and the column space $V=F^m$ is a vector space over
$K$. As such, $C=End_K(V)$. If~$\dim_K(V)>1$ then $End_K(V)$ is
not a field. Thus $\dim_K(V)=1$, so $\dim_F(V)=[K:F]=\deg(\mu_A)$,
whence $A$ is similar to the companion matrix of $\mu_A$.
\end{proof}

\begin{lemma}\label{l2} Let $A\in\gl(m)$ and let $K$ be a splitting field of $\mu_A$ over $F$. Then
$\mu_{ad\, A}$ splits over $K$. Moreover, if $S_A$ and $S_{ad\,
A}$ denote the sets of eigenvalues of $A$ and $ad\, A$ in $K$,
respectively, then $S_{ad\, A}=\{\alpha-\beta\,|\, \alpha,\beta\in
S_A\}$.
\end{lemma}

\begin{proof} According to \cite{H}, \S 4.2, we have $A=D+N$, where $D,N\in \gl(m,K)$, $D$ is
diagonalizable, $N$ is nilpotent, and $[D,N]=0$. In particular,
$\chi_D=\chi_A$. Moreover, $ad\, A= ad\; D+ad\; N$, where $ad\; D$
is diagonalizable, $ad\; N$ is nilpotent, $[ad\; D,ad\; N]=0$. As
above, $\chi_{ad\; D}=\chi_{ad\, A}$. It thus suffices to prove
the statement for $D$ instead $A$, a well-known result also found
in \cite{H}, \S 4.2.
\end{proof}

\begin{lemma}\label{gam} Let $b\in F$ and let $f\in F[X]$ be a monic polynomial of
degree $m\geq 1$. Then $C_{f(X)}\sim C_{f(X+b)}+bI$.
\end{lemma}

\begin{proof} Since the minimal polynomial of $C_{f(X+b)}$ is
$f(X+b)$, that of $C_{f(X+b)}+ b I$ is $f(X)$. As $C_{f(X)}$ and
$C_{f(X+b)}+b I$ have the same minimal polynomial, whose degree is
the size of these matrices, they must be similar.
\end{proof}

A more general result than Lemma \ref{gam}, found in \cite {GS},
reads as follows.

\begin{prop} Let $f,g\in F[X]$, where $f$
is monic of degree $m\geq 1$, and $g$ has degree $d\geq 1$ and
leading coefficient $a$. Then
\begin{equation}
\label{ind2}
 g(C_{a^{-m}}f(g(X)))\sim C_f\oplus\cdots\oplus
C_f,\quad d\text{ times}.
\end{equation}
\end{prop}

It is not difficult to verify (see \cite{CS3}) that if
$S\in\GL_m(F)$ is the Pascal matrix
$$
S=\left(
    \begin{array}{cccccc}
      1 & b & b^2 &  b^3 & \dots &   b^{m-1} \\
      0 & 1 & 2b & 3 b^2 & \dots & {{m-1}\choose{1}}b^{m-2}\\
       0 & 0 & 1 & 3 b & \dots & {{m-1}\choose{2}}b^{m-3}\\
        0 & 0 & 0 & 1 & \dots & {{m-1}\choose{3}}b^{m-4}\\
        \vdots & \vdots & \vdots &  & \ddots & \vdots\\
      0 & 0 & 0 & \dots & \dots & 1 \\
    \end{array}
  \right)
$$
then
\begin{equation}
\label{ewq} S^{-1}(C_{f(X+b)}+b I)S=C_{f(X)}.
\end{equation}
We leave it to the reader to determine if, in analogy with
(\ref{ewq}), it is possible to choose a similarity transformation
in (\ref{ind2}) that depends on $g$ but not on $f$.

%\begin{prop} Let $g\in F[X]$ have degree $d\geq 1$ and
%leading coefficient $a$. Then there exists $S\in\GL_m(F)$ such
%that for any monic polynomial $f\in F[X]$ of degree~$m$, we have
%\begin{equation}
%\label{ind4} S^{-1}g(C_{a^{-m}}f(g(X)))S= C_f\oplus\cdots\oplus
%C_f,\quad d\text{ times}.
%\end{equation}
%\end{prop}

\begin{lemma}\label{l3} Let $A\in \gl(m)$ and suppose that the centralizer $C$ of
$A$ is a subfield of $M_m(F)$. Suppose further that $ad\, A$ has
at least one non-zero eigenvalue $b\in F$. Then $F$ has prime
characteristic $p$, every $b$-eigenvector of $ad\, A$ is
invertible in $M_m(F)$, and $E_b(ad\, A)$ has the same dimension
as $C=E_0(ad\, A)$.
\end{lemma}

\begin{proof} Let $K$ be a splitting field for $\mu_A$ over $F$.
By Lemma \ref{l2} and assumption, there are eigenvalues
$\alpha,\beta\in K$ of $A$ such that $\alpha=\beta+b$ for some
$b\in F$. By Lemma \ref{l1}, $\mu_A$ is irreducible. By \cite{J},
Theorem 4.4, there is an automorphism of $K/F$ such that
$\beta\mapsto\beta+b$, where $|\mathrm{Aut}(K/F)|\leq [K:F]$ is
finite. Since $b\neq 0$, $F$ must have prime characteristic $p$.
Alternatively, $\mu_A(X)$ and $\mu_A(X+b)$ have a common root
$\beta$. Since they are irreducible in $F[X]$, they must be equal.
It follows that $\mu_A(X)=\mu_A(X+ib)$, and hence $\alpha+ib$ is a
root of $\mu_A(X)$, for every $i$ in the prime field of $F$. This
forces the prime field of $F$ to be finite.

Now $X\in E_b(ad\, A)$ if and only if
$$
AX-XA=bX,
$$
i.e.,
\begin{equation}
\label{es} AX=X(A+bI).
\end{equation}
By Lemma \ref{l1}, $A\sim C_{\mu_A(X)}$. Hence
$\mu_A(X)=\mu_A(X+b)$ and Lemma \ref{gam} imply $A\sim A+b I$.
Thus, there is $S\in\GL_m(F)$ such that $A+bI=SAS^{-1}$. Replacing
this in (\ref{es}), yields
$$
AX=XSAS^{-1},
$$
i.e.,
$$
AXS=XSA.
$$
Since $A$ is cyclic, this equivalent, by \cite{J}, \S 3.11, to
$XS\in F[A]$, or $X=f(A)S^{-1}$, for some $f\in F[X]$. But $F[A]$
is a field, so the result follows.
\end{proof}

\section{Decomposition numbers}\label{ma4}

Let $F$ be a field and let $L=\langle x\rangle$ be a 1-dimensional
Lie algebra over $F$. Let $V$ be an $L$-module of dimension $n$
and let $x_V$ be the linear operator that $x$ induces on $V$.
Suppose that $x_V$ has at least one eigenvalue in $F$ and that $V$
is an indecomposable $L$-module. This means that there is a basis
$B$ of $V$ relative to which the matrix $M_B(x_V)$ of $x_V$ is the
upper triangular Jordan block $J_n(\alpha)$, where $\alpha\in F$
is the only eigenvalue of $x_V$. %In this case we will say that $V$
%is of type $(m,\alpha)$.

Suppose next that $W$ is an indecomposable $L$-module of dimension
$m$ and that $x_W$ has eigenvalue $\beta\in F$. As above, there is
a basis $C$ of $W$ relative to which $M_C(x_V)=J_m(\beta)$.

As usual, we may view $V\otimes W$ as an $L$-module via
$$
x(v\otimes w)=xv\otimes w+v\otimes xw,\quad v,w\in V.
$$
Let $x_{V\otimes W}$ be the linear operator that $x$ induces on
$V\otimes W$. It is easy to see that the minimal polynomial of
$x_{V\otimes W}$ splits in $F$ and a single eigenvalue, namely
$\alpha+\beta$. This follows from the well-known formula:
\begin{equation}
\label{car} (x-(\alpha+\beta)\cdot 1)^k(v\otimes
w)=\underset{0\leq i\leq k}\sum{{k}\choose{i}}(x-\alpha\cdot
1)^{k-i}(v)\otimes (x-\beta\cdot 1)^i(w).
\end{equation}
\begin{question}\label{q3} How does $V\otimes W$ decompose as a direct sum of
indecomposable $L$-modules?
\end{question}
That is, what is the length $\ell$ of $V\otimes W$ and what are
the decomposition numbers $d_1\geq \cdots\geq d_\ell$ such that
$$
x_{V\otimes W}\sim J_{d_1}(\alpha+\beta)\oplus\cdots\oplus
J_{d_\ell}(\alpha+\beta)?
$$
Replacing $x_V$ by $x_V-\alpha\cdot 1$, $x_W$ by $x_W-\beta\cdot
1$, and $x_{V\otimes W}$ by $x_{V\otimes W}-(\alpha+\beta)\cdot
1$, we see that $\ell$ and the decomposition numbers $d_1\geq
\cdots\geq d_\ell$ are independent of $\alpha$ and $\beta$, and
can be computed when $\alpha=0=\beta$.

When $F$ has characteristic 0 and $n\leq m$ then $\ell=n$, with
decomposition numbers
$$
m+n-1,m+n-3,\dots,m-n+3,m-n+1.
$$
This can obtained by imbedding $L$ into $\sl(2)$ and using the
Clebsch-Gordan formula \cite{H}, \S 22.5. %These decomposition
%numbers are used in \cite{CS} to classify all uniserial modules
%for a class of solvable Lie algebras.

The analogue of Question \ref{q3} for a cyclic $p$-group and $F$
of prime characteristic~$p$ was solved by B. Srinivasan \cite{S}.
The answer is given recursively, rather than as a closed formula.
Alternative algorithms can be found in \cite{Ra} and \cite{Re}. Presumably,
Srinivasan's results translate to our present set-up mutatis
mutandis.

This section furnishes a closed formula in answer to Question
\ref{q3}, albeit only in the special case $m=p^e$, where $F$ has
prime characteristic $p$ and $e\geq 0$, as required in the proof
of Theorem \ref{uno}.

%The corresponding problem for the linear operator $x_V\otimes
%x_W$, rather than $x_{V\otimes W}=x_V\otimes 1_W+ 1_V\otimes x_W$,
%was recently solved in \cite{I}.Information on the length of
%$V\otimes W$ for the group version of this problem and in the case
%of cyclic groups of order $p$, or products of these, can be found
%in \cite{GM}, \S2.

\begin{lemma}\label{fin0} Let $p$ be a prime and let $e\geq 0$.
Then $p|{{p^e}\choose{i}}$ for any $0<i<p^e$.
\end{lemma}

\begin{proof} Since $i{{p^e}\choose{i}}=p^e{{p^e-1}\choose{i-1}}$ and the largest power of $p$ dividing $i$ is at most
$p^{e-1}$, it follows that $p$ divides ${{p^e}\choose{i}}$.
\end{proof}

\begin{prop}\label{fin} Let $F$ be a field of prime characteristic $p$ and let $e\geq 0$.
Let $L=\langle x\rangle$ be a 1-dimensional Lie algebra over $F$.
Let $V$ and $W$ be indecomposable $L$-modules of dimensions $n$
and $p^e$, respectively, where $n\leq p^e$. Suppose that $x$ has
eigenvalues $\alpha,\beta\in F$ when acting on $V$ and $W$,
respectively. Then the $L$-module $V\otimes W$ decomposes as the
direct sum of $n$ isomorphic indecomposable $L$-modules, each of
which has dimension $p^e$ and is acted upon by $x$ with a single
eigenvalue $\alpha+\beta$. In symbols,
$$
\ell=n\text{ and }d_1=\dots=d_{p^e}=p^e.
$$
\end{prop}

\begin{proof} As mentioned above, we may assume without loss of
generality that $\alpha=0$ and $\beta=0$, and we will do so,
mainly for simplicity of notation. For the same reason, we let
$m=p^e$.

Let $B=\{v_1,\dots,v_n\}$ and $C=\{w_1,\dots,w_m\}$ be bases of
$V$ and $W$ relative to which $M_B(x_V)=J_n(0)$ and
$M_C(x_W)=J_m(0)$.

Since $n\leq m$, we have
$$x_V^m=0\text{ and }x_W^m=0.
$$
Therefore, Lemma \ref{fin0} and (\ref{car}) imply
$$
x_{V\otimes W}^m=0.
$$
We next view $M=V\otimes W$ as a module for the polynomial algebra
$F[X]$ via $x_{V\otimes W}$. We wish to show that $M$ has
elementary divisors $X^m,\dots,X^m$, with multiplicity~$n$.

It follows from (\ref{car}) that
$$
x^{m-1}(v_1\otimes w_m)=v_1\otimes w_1\neq 0.
$$
Let $N_1$ be the $F[X]$-submodule of $M$ generated by $v_1\otimes
w_m$. Then $N_1$ has a single elementary divisor, namely $X^m$.

Suppose that $1\leq i<n$ and the $F[X]$-submodule of $M$, say
$N_i$, generated by $v_1\otimes w_m,\dots,v_i\otimes w_m$ has
elementary divisors $X^m,\dots,X_m$, with multiplicity~$i$. Using
(\ref{car}) we see that $v_{i+1}\otimes w_1$ appears in
$x^{m-1}(v_{i+1}\otimes w_m)$ with coefficient~1. Since
$v_{i+1}\otimes w_1\notin N_i$, the minimal polynomial of the
vector $v_{i+1}\otimes w_m+N_i\in M/N_i$ is~$X^m$. The theory of
finitely generated modules over a principal ideal domain implies
that the $F[X]$-submodule of $M$ generated by $v_1\otimes
w_m,\dots,v_i\otimes w_m,v_{i+1}\otimes x_m$ has elementary
divisors $X^m,\dots,X_m$, with multiplicity $i+1$. The result now
follows.
\end{proof}

\begin{note} Unlike what happens in characteristic 0, the
decomposition of $V\otimes W$ for $L$ is not, in general, the same
as for $\sl(2)$. Indeed, suppose $F$ has characteristic 2 and let
$V=W$ be the natural module for $\sl(2)$. Then $V\otimes W$ is an
indecomposable $\sl(2)$-module, but decomposes as the direct sum
of two indecomposable $L$-modules for $L=\langle x\rangle$, where
$x,h,y$ is the standard basis of $\sl(2)$.
\end{note}

Resuming our prior discussion, let $F$ be a field and let
$L=\langle x\rangle$ be a 1-dimensional Lie algebra over $F$. Let
$V_1,\dots,V_s$ be $L$-modules with bases $B_1,\dots,B_s$ relative
to which $M_{B_i}(x_{V_i})=J_{m_i}(\alpha_i)$, where $1\leq i\leq
s$ and $\alpha_i\in F$. Consider the $L$-module
$V=V_1\oplus\cdots\oplus V_s$. We may view $\gl(V)$ as an
$L$-module via:
$$
x\cdot f=x_V f-fx_V.
$$
Thus $x$ acts on $\gl(V)$ via $ad\, x_V$. By Lemma \ref{l2}, the
eigenvalues of $ad\, x_V$ are $\alpha_i-\alpha_j$, where $1\leq
i,j\leq s$. We can view $\gl(V)$ as the direct sum of the
$L$-submodules
$$
\mathrm{Hom}(V_j,V_i)\cong V_j^*\otimes V_i,\quad 1\leq i,j\leq s.
$$
Here $V_j^*$ is an indecomposable $L$-module upon which $x$ acts
with eigenvalue $-\alpha_j$. It is then clear that the generalized
eigenspace of $ad\, x_V$ for a given eigenvalue $\gamma$ is the
sum of all $\mathrm{Hom}(V_j,V_i)$ such that
$\alpha_i-\alpha_j=\gamma$. %Thus, the degree of the minimal
%polynomial of the restriction of $ad\, x_V$ to the generalized
%eigenspace with eigenvalue $\gamma$ is the largest of the
%decomposition numbers of all $V_j^*\otimes V_i$ such that
%$\alpha_i-\alpha_j=\gamma$.

\begin{cor}\label{co1} Keep the above notation and suppose, further, that $F$ has prime characteristic $p$ and
all $L$-modules $V_i$ have the same dimension $p^e$, for some
$e\geq 0$. Let $S=\{\alpha_i-\alpha_j\,|\, 1\leq i,j\leq s\}$, the
set of distinct eigenvalues of $ad\, x_V$. Then the minimal
polynomial of $ad\, x_V$ is $\underset{\gamma\in
S}\Pi(X-\gamma)^{p^e}$. Moreover, the elementary divisors of $ad\,
x_V$ are $(X-\gamma)^{p^e}$, $\gamma\in S$, with multiplicity
$p^em(\gamma)$, where
$$
m(\gamma)=|\{(i,j)\,|\, 1\leq i,j\leq s,\;
\alpha_i-\alpha_j=\gamma\}|.
$$
\end{cor}

\section{Proof of Theorem \ref{uno}}
\label{ma3}

Suppose conditions (C4)-(C6) hold. Then $q$ is irreducible, since
so is $h$. Let $K$ be a splitting field for $h$ over $F$. Since
$q$ is separable, the number of distinct roots of $h$ in $K$ is
exactly $p^n$.

Let $\beta\in K$ be a root of $h$. We readily verify that
$\beta+b$ is a root of $h$ for every $b\in \EF_{p^n}$. It follows
that $\beta+b$, $b\in \EF_{p^n}$, are the distinct roots of $h$ in
$K$ and each has multiplicity $p^e$. By Lemma \ref{l2}, $\mu_{ad\,
A}$ splits in $K$ and the set of eigenvalues of $ad\, A$ is
precisely $\EF_{p^n}$. Moreover, by Lemma \ref{l1}, the
centralizer of $A$ is a subfield of $M_m(F)$. In particular,
conditions (C1)-(C3) hold.

Furthermore, by Lemma \ref{l3}, all eigenvectors of $ad\, A$ are
invertible in $M_m(F)$ and all eigenspaces of $ad\, A$ have
dimension $p^{n+e}$. Thus, the sum of the dimensions of all
eigenspaces of $ad\, A$ is $p^{2n+e}$. This equals the dimension
of $\gl(m)$, namely $m^2=p^{2(n+e)}$, if and only if $e=0$.
Therefore, $ad\, A$ is diagonalizable if and only if $h$ is
separable.

Regardless of whether $e=0$ or not, we claim that the invariant
factors of $ad\, A$ are
$X^{p^{n+e}}-X^{p^e},\dots,X^{p^{n+e}}-X^{p^e}$, with multiplicity
$p^{n+e}$. For this purpose, we may assume without loss of
generality that $F=K$. Hence $\mu_{ad\, A}$ splits in $F$, by
Lemma \ref{l2}. Thus $A$ is similar to the direct sum of the
companion matrices to $X^{p^e}-\beta^{p^e}=(X-\beta)^{p^e}$, as
$\beta$ runs through the $p^n$ distinct roots of $h$ in $F$.
Hence, we may assume without loss of generality that $A$ is the
direct sum of the $p^n$ Jordan blocks $J_{p^e}(\beta)$, with
$\beta$ as above. Now, the eigenvalues of $ad\, A$ form the
subfield $\EF_{p^n}$ of $F$. Moreover, for each $b\in \EF_{p^n}$,
we have
$$
|\{(\beta_1,\beta_2)\,|\, h(\beta_1)=0=h(\beta_2),\,
\beta_1-\beta_2=b\}|=p^n.
$$
It follows from Corollary \ref{co1} that the elementary divisors
of $ad\, A$ are $(X-b)^{p^e}$, $b\in \EF_{p^n}$, each with
multiplicity $p^{n+e}$. Since $\underset{b\in
\EF_{p^n}}\Pi(X-b)=X^{p^n}-X$, the claim follows.

Suppose conversely that (C1)-(C3) hold. Since the eigenvalues of
$ad\, A$ form a subfield of $F$, we see that $ad\, A$ has a
non-zero eigenvalue in $F$. It follows from Lemma \ref{l3} that
(C4) holds. Since the centralizer of $A$ is a subfield of
$M_m(F)$, Lemma \ref{l1} shows that (C5) holds.

Let $K$ be a splitting field for $\mu_A$ over $F$. Let $q$ be the
separable part of $\mu_A$, so that $\mu_A(X)=q(X^{p^e})$ for some
$e\geq 0$.

Let $\EF_{p^n}$ be the subfield of $F$ formed by the roots of
$ad\, A$ ensured by (C3). Let $\beta\in K$ be a root of $\mu_A$.
It follows from Lemma \ref{l2} that the distinct roots of $\mu_A$
in $K$ are $\beta+b$, for all $b\in \EF_{p^n}$. Since $\mu_A$ and
its separable part must have the same number of distinct roots, we
deduce that $q$ has degree $p^n$. Let $\alpha=\beta^{p^e}\in K$.
Then $\alpha$ is a root of $q$. Moreover, if $c\in \EF_{p^n}$ then
$\alpha+c$ is also a root of $q$. Indeed, $b\mapsto b^{p^e}$ is an
automorphism of $\EF_{p^n}$, so $c=b^{p^e}$ for some $b\in
\EF_{p^n}$. Therefore,
$$
q(\alpha+c)=\mu_A(\beta+b)=0.
$$
Thus $F[\alpha]$ is a splitting field for $q$ over $F$. Since $q$
is separable, we deduce that $F[\alpha]/F$ is a finite Galois
extension, whose Galois group we denote by $G$. We claim that
$\alpha^{p^n}-\alpha\in F$. To see this, it suffices to show that
$\alpha^{p^n}-\alpha\in F$ is fixed by every $\sigma\in G$. Let
$\sigma\in G$. Then $\sigma(\alpha)$ must be a root of $q$, so
$\sigma(\alpha)=\alpha+b$ for some $b\in \EF_{p^n}$. Therefore,
$$
\sigma(\alpha^{p^n}-\alpha)=(\alpha+b)^{p^n}-(\alpha+b)=\alpha^{p^n}-\alpha,
$$
as required. Thus $\alpha^{p^n}-\alpha=a\in F$, so $X^{p^n}-X-a\in
F[X]$ has $\alpha$ as root. Hence $q|(X^{p^n}-X-a)$. Since these
polynomials have the same degree and are monic, they must be
equal. This completes the proof of the theorem.

\section{Primitive elements of intermediate fields in a Galois extension}
\label{alfaH}

Let $F$ be a field of prime characteristic $p$ having $\EF_{p^n}$
as a subfield and consider the generalized Artin-Schreier
polynomial $q=X^{p^n}-X-a\in F[X]$. Let $\alpha$ be a root of $q$
in a field extension of $F$. Then $\alpha+b$, $b\in \EF_{p^n}$,
are all roots of $q$, so $F[\alpha]$ is a splitting field for $q$
over $F$. Moreover, all roots of $q$ have the same degree over
$F$, since $F[\alpha]=F[\alpha+b]$ for any $b\in \EF_{p^n}$. Thus,
all irreducible factors of $q$ in $F[X]$ have the same degree. Let
$G$ be the Galois group of $F[\alpha]/F$. We claim that $G$ is
elementary abelian $p$-group. Indeed, let $\sigma,\tau\in G$. Then
$\sigma(\alpha)=\alpha+b$ and $\tau(\alpha)=\alpha+c$ for some
$b,c\in \EF_{p^n}$. Therefore $\sigma^p=1$ and
$\sigma\tau=\tau\sigma$, as claimed. Since $|G|=[F[\alpha]:F]$,
which is the degree the minimal polynomial of $\alpha$ over $F$,
we deduce that all irreducible factors of $q$ in $F[X]$ have
degree $p^m$ for a unique $0\leq m\leq n$

If $a=0$ it is obvious that $m=0$. If $a\neq 0$ and $F=\EF_{p^n}$
then $m=1$. More generally, if $F$ algebraic over $\EF_{p^n}$ and
there is no $b\in F$ such that $q(b)=0$ then $m=1$. Indeed, let
$f$ be the minimal polynomial of $\alpha$ over $F$ and let $p^m$
be its degree. By assumption $m>0$.  Let $E$ be the subfield of
$F$ obtained by adjoining $a$ and the coefficients of $f$ to
$\EF_{p^n}$. Then $f$ is irreducible and a factor of $q$ in
$E[X]$. By above, $\mathrm{Gal}(E[\alpha]/E)$ is an elementary
abelian group of order $p^m$. Since $E$ is a finite field,
$\mathrm{Gal}(E[\alpha]/E)$ is cyclic, so $m=1$.

%For instance, let $\EF_4=\EF_2[s]$, where $s^2+s+1=0$. Then
%$X^4-X-1\in \EF_4[X]$ factors as follows:
%$$
%X^4-X-1=(X^2+X+s)(X^2+X+s^2).
%$$

Assume henceforth that $q$ is actually irreducible. Then
$K=F[\alpha]$ is a splitting field for $q$ over $F$ and
$G=\mathrm{Gal}(K/F)$ is an elementary abelian $p$-group of order
$p^n$. More explicitly, for $b\in \EF_{p^n}$, let $\sigma_b\in G$
be defined by $\sigma_b(\alpha)=\alpha+b$. Then $b\mapsto
\sigma_b$ defines a group isomorphism $\EF_{p^n}^+\to G$. In
particular, $G$ has normal subgroups of all possible orders.

Suppose that $m$ satisfies $0\leq m\leq n$. Let $H$ be a subgroup
of $G$ of order $p^{m}$. Then the fixed field $E=K^H$ of $H$
satisfies $[K:E]=p^m$. Since $F[\alpha]=E[\alpha]$, the minimal
polynomial $\mu_{\alpha,E}$ of $\alpha$ over $E$ must have degree
$p^m$. In fact,
$$
\mu_{\alpha,E}(X)=\underset{\sigma\in H}\Pi(X-\alpha^\sigma).
$$
Since $\mu_{\alpha,E}(X)$ divides $q=X^{p^n}-X-a\in E[X]$, it
follows that all irreducible factors of $q$ in $E[X]$ have degree
$p^m$. In fact,
$$
X^{p^n}-X-a=\underset{\sigma\in\mathrm{Gal}(E/F)}\Pi\mu_{\alpha,E}(X)^{\sigma}.
$$
Let
\begin{equation}
\label{al3} \alpha_H=\underset{\sigma\in H}\Pi\alpha^\sigma.
\end{equation}
Since every $\sigma\in G$ is of the form $\sigma_b$ for $b\in
\EF_{p^n}$, it follows that $\alpha_H$ is a monic polynomial in
$\alpha$ of degree $p^m$ with coefficients in $\EF_{p^n}$. Since
$\alpha$ has degree $p^n$ over $F$, the degree of $\alpha_H$ over
$F$ is at least $p^{n-m}$. But clearly $\alpha_H\in E$, where
$[E:F]=p^{n-m}$. It follows that $E=F[\alpha_H]$.

As just noted, $\alpha_H$ as an $\EF_{p^n}$-linear combination of
powers of~$\alpha$. In fact, we may use the so-called Dickson
invariants, found by L. E. Dickson \cite{D} in 1911, to obtain a
sharper result. These invariants have been revisited numerous
times (see, for instance, \cite{H2} and \cite{St}).

Consider the polynomial $\Phi_m$ in the polynomial algebra
$F[A,B_1,\dots,B_m]$, defined as follows:
\begin{equation}
\label{al4} \C_m(A,B_1,\dots,B_m)=\underset{s_1, \dots,s_m\in
\EF_p}\Pi(A+s_1B_1+\cdots+s_mB_m).
\end{equation}
Clearly $\Phi_m$ is $\GL_m(\EF_p)$-invariant. Dickson showed that
$$
\Phi_m=A^{p^m}+f_{m-1}(B_1,\dots,B_m)A^{p^{m-1}}+\cdots+f_{1}(B_1,\dots,B_m)A^{p}+f_{0}(B_1,\dots,B_m)A,
$$
where $f_0,\dots,f_{m-1}\in \EF_p[B_1,\dots,B_m]$ are
algebraically independent and generate
$\EF_p[B_1,\dots,B_m]^{\GL_m(\EF_p)}$. Moreover, $\Phi_m$, and
hence $f_{m-1},\dots,f_0$, can be recursively computed from
\begin{equation}
\label{rec0} \Phi_0=A,
\end{equation}
\begin{equation}
\label{rec1}
\C_i=\C_{i-1}(A,B_1,\dots,B_{i-1})^p-\C_{i-1}(B_i,B_1,\dots,B_{i-1})^{p-1}\C_{i-1}(A,B_1,\dots,B_{i-1}).
\end{equation}

Let $\sigma_{b_1},\dots,\sigma_{b_m}$ be generators of $H$. This
means that $\sigma_{b_1},\dots,\sigma_{b_m}$ are in $H$ and that
$b_1,\dots,b_m$ are linearly independent over $\EF_p$. It follows
from (\ref{al3}) that
$$
\alpha_H=\underset{s_1,\dots,s_m\in
\EF_p}\Pi(\alpha+s_1b_1+\cdots+s_m
b_m)=\Phi_m(\alpha,b_1,\dots,b_m).
$$
This, together with (\ref{rec0}) and (\ref{rec1}) allows us to
recursively find $c_0,\dots,c_{m-1}\in \EF_{p^n}$ such that
$$
\alpha_H=\alpha^{p^m}+c_{m-1}\alpha^{p^{m-1}}+\cdots+c_{1}\alpha^{p}+c_0\alpha.
$$

In certain special cases we actually  have $c_0,\dots,c_{m-1}\in
\EF_p$, in which case we will say that $H$ has property $P$.

Let $R$ be the subgroup of $\EF_{p^n}^+$ that corresponds to $H$
under $\EF_{p^n}^+\to G$. Thus, $R$ is an $\EF_p$-subspace of
$\EF_{p^n}$, namely the $\EF_p$-span of $b_1,\dots,b_m$.

\begin{lemma} $H$ has property $P$ if and only $R$ is invariant under the Frobenius
automorphism of $\EF_{p^n}$, in which case
$$
\alpha_H=f_R(\alpha),
$$
where
$$
 f_R(Y)=\underset{b\in R}\Pi(Y+b).
$$
\end{lemma}

\begin{proof} We start by showing that $H$ has property $P$ if and only if
$f_R\in \EF_{p^n}[Y]$ has coefficients in $\EF_p$.

Suppose first that there exist $c_{m-1},\dots,c_0\in \EF_p$ such
that
$$
\alpha_H=\alpha^{p^m}+c_{m-1}\alpha^{p^{m-1}}+\cdots+c_1\alpha^p+c_0\alpha.
$$
Set
\begin{equation}
\label{efe} f(Y)=Y^{p^m}+c_{m-1}Y^{p^{m-1}}+\cdots+c_1 Y^p+c_0Y\in
\EF_p[Y].
\end{equation}
Let $b\in R$. Since $\alpha_H^{\sigma_b}=\alpha_H$, it follows
that
$$
f(b)=0.
$$
Therefore $f_R=f\in \EF_p[Y]$. Conversely, if $f_R\in \EF_p[X]$
then
$$
\alpha_H=\underset{b\in R}\Pi(\alpha+b)=f_R(\alpha),
$$
which is an $\EF_p$-linear combination of
$\alpha^{p^m},\dots,\alpha^p,\alpha$ with first coefficient 1.

Let $\tau$ be the Frobenious automorphism of $\EF_{p^n}$. Then
$f_R\in \EF_p[Y]$ if and only if
$$
f_R(Y)=f_R(Y)^\tau=\underset{b\in R}\Pi(Y+b^\tau )=\underset{b\in
R^\tau}\Pi (Y+b)=f_{R^\tau}(Y),
$$
which is equivalent to $R=R^\tau$.
\end{proof}

Suppose that $R$ is actually a subfield of $\EF_{p^n}$. Then $R$
is certainly invariant under $b\mapsto b^p$. Moreover,
$f_R(Y)=Y^{p^m}-Y$. Therefore, in this case,
$$
\alpha_H=\alpha^{p^m}-\alpha.
$$
In particular, $\alpha_G=a$.

\begin{cor}\label{indep2} If $b_1,\dots,b_m\in \EF_{p^m}$ are linearly independent over $\EF_p$ then
$$
f_0(b_1,\dots,b_m)=-1,\text{ whereas }f_j(b_1,\dots,b_m)=0,\text{
if }1\leq j\leq m-1.
$$
\end{cor}

Corollary \ref{indep2} is not true, in general, if
$b_1,\dots,b_m\in R$ are linearly dependent over~$\EF_p$, as the
case $m=2$ will confirm by taking $b_1=1=b_2$ for $j=0,1$.

\begin{exa}{\rm Here we furnish examples of subspaces $R$ of $\EF_{p^n}$
that are invariant under $b\mapsto b^p$ but are not subfields of
$\EF_{p^n}$, even when $m|n$.

Suppose first $m=1$, where $(p-1)|n$ and $p$ is odd. Take $c\in
\EF_p$, $c\neq 0,1$, and set
$$
f(Y)=Y^p-cY\in \EF_p[Y].
$$
Since $Y^{p^{p-1}}\equiv Y\mod f$, it follows that $f$ splits in
$\EF_{p^{p-1}}$ and hence in $\EF_{p^n}$. The roots of $f$ in
$\EF_{p^n}$ form a $1$-dimensional, Frobenius-invariant, subspace
$R$ of $\EF_{p^n}$, which is not a field, since $\EF_p$ is the
only $1$-dimensional subfield of $\EF_{p^n}$ and 0 is the only
element of $\EF_p$ that is a root of $f$.

Suppose next that $m=2$, where $2p|n$ and $p$ is odd. Let
\begin{equation}
\label{terra} f(Y)=\underset{j\in \EF_p}\Pi(Y^p-Y-j).
\end{equation}
Thus, $f$ is the product of all Artin-Schrier polynomials in
$\EF_p[Y]$. This readily implies that the roots of $f$ form a
2-dimensional, Frobenius-invariant, subspace $R$ of $\EF_{p^n}$
containing $\EF_p$. In particular, $f(cY)=f(Y)$ for all $0\neq
c\in \EF_p$. Since $Y|f(Y)$ and $f(Y)$ has degree $p^2$, it
follows that
$$
f(Y)=Y^{p^2}+a Y^{p}+b Y,
$$
where $a,b\in \EF_p$. Using (\ref{terra}) and $p>1$ to compute
$a,b$ reveals that
$$
f(Y)=Y^{p^2}-Y^{p}-Y.
$$
However, the subfield of $\EF_{p^n}$ obtained by adjoining $R$ to
$\EF_p$ is $\EF_{p^p}$. Since $p>2$, it follows that $R$ is not a
subfield of $\EF_{p^n}$. }
\end{exa}

\section{Existence of irreducible generalized Artin-Schreier
polynomials} \label{hector}

We begin this section by giving an elementary example of an
irreducible Artin-Schreier polynomial. We then furnish
substantially more general examples, which require the use of
preliminary results from \cite{G} and \cite{MS}. In fact, we give
necessary and sufficient conditions for a polynomial
$$q(X)=X^{p^{n+e}}-X^{p^e}-g(Z^r)\in F[X]$$
to be irreducible, where $X$ and $Z$ are algebraically independent
elements over an arbitrary field $K$ of prime characteristic $p$,
$n>0$, $r>0$, $e\geq 0$, $F=K(Z)$, and $g(Z)\in K[Z]$ is non-zero of degree
relatively prime to $p$.

Recall that an element $\pi$ of and  an integral domain $D$ is
irreducible if $d$ is neither~0 nor a unit and whenever $\pi=ab$
with $a,b\in D$ then $a$ or $b$ is a unit.

It is easy to see that $X^{p^n}-X-Z$ is irreducible in $K[X][Z]$,
hence in $K[X,Z]$ and therefore in $K[Z][X]$. It follows from
Gauss' Lemma (see \cite{J}, \S 2.16) that $X^{p^n}-X-Z$ is
irreducible in $F[X]$.

\begin{theorem}\label{nhg1} {\rm (see \cite{MS}, Proposition 1.8.9)}. Let $D$ be an integral domain. Let $X, Z$ be algebraically independent elements over
$D$. Let $f\in D[X]$ and $g\in D[Z]$. If $\gcd(deg(f),
\deg(g))=1$, then $h(X, Z)= f(X)-g(Z)$ is irreducible in  $D[X,
Z]$.
\end{theorem}

\begin{theorem}\label{nhg2}  {\rm (see \cite{G}, Theorem 1.1)}.
Let $D$ be a unique factorization domain. Let~$t$ be any positive
integer and let~$f$ be an irreducible polynomial in $D[Z]$ of
positive degree $m$, leading coefficient $a$ and nonzero constant
term $b$. Suppose that for each prime $p$ dividing $t$ and any
unit $u$ of $D$ at least one of the two following statements is
true:

\noindent {\rm (A)} \,$ua\notin D^p$;

\noindent {\rm (B)} \,{\rm (i)} \,$(-1)^mub\notin D^p$ \,and\,
{\rm (ii)} \,$ub\notin D^2$, if $4|t$.

%\vspace{0.1cm}
 \noindent {\it Then
 $$\text{$f(Z^t)$ is irreducible in
 $D[Z]$.}$$
}
\end{theorem}

\begin{theorem}\label{nhg3} {\rm (see \cite{G}, Corollary 4.6 (b))}. Let $D$ be a unique factorization domain of prime characteristic $p$. Let $f(Z)\in D[Z]$ be an arbitrary
polynomial of positive degree that is  irreducible in $D[Z]$, and
let $s$ be any positive integer.  Then
$$\text{$f\bigl(Z^{p^s}\bigr)$ \!is reducible \!in
$D[Z]\!\!\!\iff$\!\!\! there exists a unit $u$ of $D$ such that
$uf(Z)\!\in \!D^p[Z]$.}$$
\end{theorem}

We can now prove the following result.

%\smallskip
\begin{theorem} Let $K$ be a field of prime characteristic  $p$. Let $X,Z$ be algebraically independent elements over
$K$ and set $F=K(Z)$. Let $n, e, r$ be integers such that $n>0,
r>0$ and $e\geq 0$. Let $g(Z)=c_d Z^d+\cdots+c_1 Z+c_0\in K[Z]$ be
any polynomial whose degree $d$ is coprime with $p$. Then
$h(X,Z)=X^{p^{n+e}}-X^{p^e}-g(Z^r)$ is irreducible in $F[X]$ if
and only if at least one of the following conditions is satisfied:
\begin{center}{\rm (i)} $p\nmid r$, \,\,{\rm (ii)} $e =0$, \,\,{\rm (iii)} $g(Z)\notin K^p[Z]$.\end{center}
\end{theorem}

\begin{proof} By unique factorization in $\mathbb Z$ there exist a
positive integer~$r_0$ coprime with~$p$ and a nonnegative integer
$s$ such that $r=r_0p^s$, so $p|r$ if and only if $s\ge 1$.

Suppose first none of the conditions (i)-(iii) is fulfilled. Thus
$e\ge 1$, $s\ge 1$ and $g(Z)\in K^p[Z]$. The last two conditions
imply $g(Z^r)=Q^p(Z)$ for some $Q(Z)$ in $K[Z]$, and now the first
condition implies that $h(X,Z)$ is a $p$th power in $F[X]$.

Suppose next that at least one of the conditions (i)-(iii) holds.
We wish to show that $h(X,Z)$ is irreducible in $F[X]$.

\smallskip
Case (i).  Suppose $p\nmid r$. Therefore $p\nmid dr$. Since
$g(Z^{r})$ has degree~ $dr$, from  the case $D=K$, $f(X)=
X^{p^{n+e}}-X^{p^e}$ of Theorem \ref{nhg1}, with $g(Z^{r})$
instead of $g(Z)$, we see that $h(X, Z)$ is irreducible in $K[X,
Z]$, hence in $K[Z][X]$, and therefore in $F[X]$ by Gauss' Lemma.

\smallskip
Case (ii). Suppose $e=0$. The previous case guarantees that
$X^{p^n}-X-g(Z^{r_0})$ is irreducible in $F[X]$, so we can suppose
$s\ge 1$. Let $D=K[X]$. Therefore $f(Z)=X^{p^n}-X-g(Z^{r_0})$ is
an irreducible polynomial in $D[Z]$ of degree $m=dr_0$ and
constant term $b=X^{p^n}-X-c_0$. Since $X^{p^n}-X-c_0$ has no
repeated roots, for any $u\in K^*$ (i.e., for any unit $u$ of $D$)
we have
 $(-1)^{m}ub\notin D^p$ as well as $ub\notin D^2$ if $4|p^s$ (i.e., if $p=2$ and $s\ge 2$).
Thus part (B) of  Theorem \ref{nhg2} is satisfied with $D=K[X]$
and $t=p^s$. We conclude that $f(Z^{p^s})$ is irreducible in
$D[Z]$, and therefore in $K[Z][X]$. Hence $h(X,Z)=f(Z^{p^s})$ is
irreducible in  $F[X]$ by Gauss' Lemma.

\smallskip
Case (iii). Assume $g(Z)\notin K^p[Z]$. From the cases (i) and
(ii) we can assume $s\ge 1$ and $e\ge 1$. Suppose, if possible,
that $h(X, Z)$ is reducible in $F[X]$. Letting $D=K[X]$ we get,
via Gauss' Lemma, that $h(X, Z)$ is reducible in $K[Z][X]$, and
therefore in $D[Z]$. Letting $f(Z)=X^{p^n}-X-g(Z^{r_0})$ we get
that $f(Z^{p^s})= h(X, Z)$ is reducible in $D[Z]$. But, as seen
above, $f(Z)=X^{p^n}-X-g(Z^{r_0})$ is irreducible in $D[Z]$ of
positive degree $m=dr_0$. Hence, by Theorem \ref{nhg3}, there
exists a unit $u$ of $D$ and $Q\in D^p[Z]$ such that $uf(Z)=Q(Z)$.
In other words, there exist $u\in K^*$ and $Q_0, Q_1, \dots ,
Q_m\in K[X]$ such that

\vspace{-0.3cm}
\begin{equation*}
u\biggl(X^{p^{n+e}} - X^{p^e} - \sum_{0\le k\le
d}c_kZ^{kr_0}\biggr) = \sum_{0\le k\le m} Q_k^p Z^k.
\end{equation*}
Equating coefficients of like monomials we obtain
$$\text{$u(X^{p^{n+e}} - X^{p^e}-c_0) = Q_0^p$ and $uc_k= Q_{kr_0}^p\in K^p$ for $k=1, \dots , d$.}$$
Since $Q_0$ has degree $m_0= p^{n+e-1}$, there must exist $d_0,
d_1, \dots , d_{m_0}$ in $K$ such that $Q_0=  \sum_{0\le k\le m_0}
d_kX^k$, whence
\begin{equation}
\label{nhg4}
u(X^{p^{n+e}} - X^{p^e}) -uc_0 = \sum_{0\le k\le m_0}d_k^pX^{kp}.
\end{equation}
Equating leading coefficients yields $u = d_{m_0}^p\in K^p$, so
$c_k\in K^p$ for $k=1, \dots , d$.  But
$$u(X^{p^{n+e}} - X^{p^e})= \biggl(d_{m_0}X^{p^{e-1}}(X-1)^{p^{n-1}}\biggr)^{\!\!p},$$
so (\ref{nhg4}) gives
$$uc_0= u(X^{p^{n+e}} - X^{p^e})- Q_0^p= \biggl(d_{m_0}X^{p^{e-1}}(X-1)^{p^{n-1}} -Q_0\biggr)^{\!\!p}\in K[X]^p\cap K=K^p,$$
and a fortiori $c_0\in K^p$. Hence all $c_i\in K^p$, against the
fact that $g(Z)\notin K^p[Z]$.
\end{proof}

\noindent{\bf Acknowledgments.} We are grateful to D. Stanley for
fruitful conversations, R. Guralnick and A. Zalesski for useful
references, and the referee for valuable comments.

%\bibliographystyle{amsplain}
%==================================================


\begin{thebibliography}{FH}
%==================================================


%\bibitem[B]{B}  Bourbaki, N.,
%\emph{Lie Groups and Lie Algebras: chapters 1--3},
%Springer-Verlag, Berlin, 1998.


%\bibitem[CS]{CS} L. Cagliero and F. Szechtman, \emph{Jordan-Chevalley decomposition in finite dimensional Lie algebras}, Proc. Amer. Math. Soc. 139 (2011)
%3909–-3913.

\bibitem[CS]{CS} L. Cagliero and F. Szechtman, \emph{The classification of uniserial $\sl(2)\ltimes V(m)$-modules and a new interpretation of the Racah-Wigner 6j-symbol}, J. Algebra 386 (2013)
142-–175.

\bibitem[CS2]{CS2} L. Cagliero and F. Szechtman, {\em  On the theorem of the primitive element with applications to the representation theory of associative and Lie algebras },
Canad. Math. Bull., to appear.


\bibitem[CS3]{CS3} L. Cagliero and F. Szechtman, {\em Classification of linked indecomposable modules
of a family of solvable Lie algebras over an arbitrary field of characteristic 0}, submitted.

%\bibitem[CS5]{CS5} L. Cagliero and F. Szechtman, {\em Indecomposable modules of 2-step solvable Lie algebras in arbitrary characteristic}, %submitted.

\bibitem[D]{D} L. E. Dickson, \emph{A Fundamental System of Invariants of the General
Modular Linear Group with a Solution of the Form Problem}, Trans.
A.M.S. 12 (1911) 75--98.

\bibitem[G]{G} N. H. Guersenzvaig, \emph{Elementary criteria for irreducibility of
$f(X^r)$},
Israel J. Math. 169 (2009) 109-–123.

%\bibitem[GM]{GM} R. Guralnick and G. Malle, \emph{Classification of 2F-modules,
%II}, Finite Groups 2003, de Gruyter, Berlin (2004) 117--183.

%\bibitem[L]{L} S. Lang {\em Algebra},
%Addison-Wesley, Reading, 1984.


%\bibitem[He]{He} M. Herschend, \emph{Solution to the Clebsch-Gordan problem for string
%algebras}, J. Pure Appl. Algebra 214 (2010) 1996-–2008.

%\bibitem[HK]{HK} K. Hoffman and R. Kunze {\em Linear Algebra},
%Prentice-Hall,  Englewood Cliffs, 1961.


\bibitem[GL]{GL} R. Gow and T. J. Laffey, \emph{J. Group Theory}, 9 (2006) 659–-672.

\bibitem[GP]{GP} I.M. Gelfand, V.A. Ponomarev, \emph{Remarks on the classification of a
pair of commuting linear transformations in a finite dimensional
vector space}, Functional Anal. Appl. \textbf{3} (1969) 325--326.



\bibitem[GS]{GS} N. Guerszenzvaig and F. Szechtman, {\em Is every matrix similar to a polynomial in a companion
matrix?},
Linear Algebra Appl. 437 (2012) 1611–-1627.


\bibitem[H]{H} J. E. Humphreys, {\em Introduction to Lie Algebras and Representation
Theorey}, Springer-Verlag, New York, 1972.

\bibitem[H2]{H2} J. E. Humphreys, \emph{Another look at Dickson's invariants for finite linear
groups}, Comm. Algebra 22 (1994) 4773-–4779.

\bibitem[I]{I} K. Iima and R. Iwamatsu, \emph{On the Jordan decomposition of tensored matrices of Jordan
canonical forms}, Math. J. Okayama Univ. 51 (2009) 133–-148.

\bibitem[Is]{Is} I. M. Isaacs, \emph{Degrees of sums in a separable field
extension}, Proc. Amer. Math. Soc. \textbf{25} (1970) 638-–641.

\bibitem[K]{K} I. Kaplansky, \emph{Fields and Rings}, Univ. of
Chicago Press, Chicago, 1969.


\bibitem[J]{J} N. Jacobson, {\em Basic Algebra I}, Freeman, New
York, 1985.

\bibitem[M]{M} I. Makedonskyi, \emph{On wild and tame finite-dimensinal Lie algebras},
Funct. Anal. Appl., 47 (2013) 271--283.

\bibitem[MS]{MS} M. Mignotte and D. \c{S}tef\u{a}nescu, {\em Polynomials, An Algorithmic
Approach},
Springer-Verlag, Singapore, 1999.

\bibitem[N1]{N1} T. Nagell, \emph{Bemerkungen $\rm{\ddot{u}}$ber zusammengesetzte Zahlk$\rm{\ddot{o}}$rper}, Avh. Norske
Vid. Akad. Oslo (1937) 1--26.

\bibitem[N2]{N2} T. Nagell, \emph{Bestimmung des Grades gewisser relativalgebraischer
Zahlen}, Monatsh. Math. Phys. \textbf{48} (1939) 61–-74.


\bibitem[Ra]{Ra} T. Ralley, \emph{Decomposition of products of modular representations},
Bull. Amer. Math. Soc. 72 (1966) 1012--1013.

\bibitem[Re]{Re} J.-C. Renaud, \emph{The decomposition of products in the modular representation ring
of a cyclic group of prime power order}, J. Algebra 58 (1979)
1--11.

\bibitem[Ro]{Ro}  W. E. Roth, \emph{On direct product matrices}, Bull. Amer. Math. Soc. 40 (1934) 461--468.


\bibitem[S]{S} B. Srinivasan, \emph{The modular representation ring of a cyclic
$p$-group}, Proc. London Math. Soc. (3) 14 (1964) 677-–688.

\bibitem[St]{St} R. Steinberg, \emph{On Dickson's theorem on
invariants}, J. Fac. Sci. Univ. Tokyo Sect. IA Math. 34 (1987)
699-–707.

\bibitem[V]{V} I. M. Vinogradov, \emph{An introduction to the theory of
numbers}, Pergamon Press, London, 1955.


\end{thebibliography}
\end{document}